\documentclass[11pt,a4paper]{article}
\usepackage[utf8]{inputenc}
\usepackage[english]{babel}
\usepackage[T1]{fontenc}
\usepackage{amsmath}
\usepackage{amsfonts}
\usepackage{amssymb}
\usepackage{amsthm}
\usepackage{color}
\usepackage{enumerate}

\newcommand{\R}{\mathbb{R}}
\newcommand{\C}{\mathbb{C}}

\newcommand{\sh}{\mathrm{sin.h}}
\newcommand{\ch}{\mathrm{cos.h}}

\title{Orbital integrals on Lorentzian symmetric spaces}
\author{\textsc{Thibaut Grouy}\footnote{\textsc{D\'epartement de Math\'ematique, Universit\'e libre de Bruxelles}\newline Campus de la Plaine, CP 218, Boulevard du Triomphe, 1050 Brussels, Belgium\newline
\emph{E-mail address} : \texttt{tgrouy@ulb.ac.be}}
}

\date{}

\newtheorem{theorem}{Theorem}
\newtheorem{lemma}[theorem]{Lemma}
\newtheorem{proposition}[theorem]{Proposition}
\newtheorem{corollary}[theorem]{Corollary}

\theoremstyle{remark}
\newtheorem{remark}{Remark}

\theoremstyle{definition}

\newtheorem{definition}{Definition}

\begin{document}

\maketitle

\begin{abstract}
In this paper, we address the problem of determining a function in terms of its orbital integrals on Lorentzian symmetric spaces.
It has been solved by S. Helgason \cite{Hel59} for even-dimensional isotropic Lorentzian symmetric spaces via a limit formula involving the Laplace-Beltrami operator. The result has been extended by J. Orloff \cite{Orl87} for rank-one semisimple pseudo-Riemannian symmetric spaces giving the keys to treat the odd-dimensional isotropic Lorentzian symmetric spaces.
Indecomposable Lorentzian symmetric spaces are either isotropic or have solvable transvection group.
We study orbital integrals including an inversion formula on the solvable ones which have been explicitly described by M. Cahen and N. Wallach \cite{CaWa70}.
\newline


\end{abstract}


\section*{Introduction}

On a pseudo-Riemannian space $(M, \hat{g})$, one defines \emph{pseudo-spherical integrals} which are parametrized by a pseudo-radial coordinate $r$ and a sign. They are denoted $M^r_\pm$ and associate to any compactly supported continuous function $f$ on $M$ the functions $(M^r_\pm f)$ on $M$ whose value at a point $x$ is given by the integrals of $f$ over pseudo-spheres centered at $x$, namely
$$\mathrm{Exp}_x (\Sigma_{\pm r^2} (x)), \quad \text{where } \Sigma_{\pm r^2} (x) := \lbrace X \in T_x M \mid \hat{g}_x (X, X) = \pm r^2 \rbrace$$
and $\mathrm{Exp}_x$ is the exponential mapping at $x$ associated to the Levi-Civita connection.
If the metric $\hat{g}$ is Lorentzian, we integrate $f$ over the connected components of the pseudo-spheres at $x$ to define its pseudo-spherical integrals at $x$.
When the space $M$ is $G$-homogeneous for a Lie subgroup $G$ of the isometry group, one also defines the \emph{orbital integrals} of $f$ at any point $x$ as the integrals of $f$ over orbits of the isotropy group $K_x$ of $x$ in $G$, with respect to an invariant measure, provided it exists.
If the pseudo-spheres centered at any point $x$ as well as the light cone $\lbrace \mathrm{Exp}_x (X) \mid X \in T_x M, \, X \neq 0, \, \hat{g}_x (X, X) = 0 \rbrace$ are orbits of $K_x$, then $M$ is said to be \emph{isotropic} ; in that case, the pseudo-spherical integrals of $f$ are equal to its orbital integrals.
\newline

If the metric $\hat{g}$ is Riemannian, only the integral operators $M^r_+$ make sense and they are called \emph{spherical integrals}. Assuming that the space $M$ is two-point homogeneous, they appear when the totally geodesic Radon transforms are composed with dual transforms. An inversion formula for these Radon transforms is obtained by applying the Laplace-Beltrami operator to the result of this composition.
The proof uses the following Darboux equation
$$L_x (M^r_+ f) (x) = L_r (M^r_+ f) (x)$$
where $L_r$ is the radial part of $L$ (see details in \cite{Hel59} for spaces of constant curvature and \cite{Hel10} for general two-point homogeneous spaces : the compact ones have been uniformly treated in \cite{Hel65a} and the non-compact ones in \cite{Rou01}).
Spherical integrals also allow to inverse the Fourier transform defined on Riemannian symmetric spaces of the non-compact type by S. Helgason in \cite{Hel65b}. Concretely, they are used to reduce the problem to Harish-Chandra's inversion formula for the spherical Fourier transform on semisimple Lie groups (see details in \cite{Hel94}, chapter III).
\newline

Orbital integrals were first defined on Lie groups : they are integrals over the conjugacy classes and they play a role in Harmonic Analysis. Indeed, in \cite{HaCh51}, Harish-Chandra gave the Plancherel formula for a Fourier transform he defined on any complex semisimple Lie group extending the result of Gelfand and Naimark \cite{GelNai48} on $\mathrm{SL} (n, \C)$.
Later, Gelfand and Graev noticed \cite{GelGra55} that the Plancherel formula for this Fourier transform on classical complex Lie groups is obtained from a limit formula which expresses the value of a function at the neutral element in terms of its orbital integrals.
In \cite{HaCh56}, Harish-Chandra gave a limit formula for the orbital integrals on any real semisimple Lie group.
Even in this case, it turns out to be useful to get a Plancherel formula. A. Bouaziz did the same for the orbital integrals on any real reductive Lie group $G$ \cite{Bou94} and P. Harinck on the quotient $G_\C / G$ where $G_\C$ is the complexified group of $G$ \cite{Har95}.
\newline

We are interested in the general problem of determining a function in terms of its orbital integrals on symmetric spaces.  Like in the group case, we expect to solve it with a limit formula.
On isotropic Riemannian symmetric spaces, the orbital integral problem is trivially solved since the orbits of the isotropy group of any point $x$ are the spheres centered at $x$ and they shrink to the point as their radius $r$ goes to zero. In other words,
$$f(x) = \underset{r \rightarrow 0^+}{\lim}\; (M^r_+ f) (x).$$
In contrast, on isotropic Lorentzian symmetric spaces, the "limit" of the pseudo-spheres as their pseudo-radius $r$ goes to zero is the light cone. However, there still exists a solution to the orbital integral problem.
Indeed, if the dimension of the space $n > 2$ is even, the expression of a function $f$ in terms of its orbital integrals is given as the following limit formula due to S. Helgason \cite{Hel59}
$$f(x) = c \; \underset{r \rightarrow 0^+}{\lim}\; r^{n-2} P (L) (M^r_+ f) (x),$$
where $P$ is a polynomial, $L$ is the Laplace-Beltrami operator associated to the metric and $c$ is a real constant.
J. Orloff generalized this limit formula to semisimple pseudo-Riemannian symmetric spaces of rank one \cite{Orl87}.
In this paper, we address the orbital integral problem on all other indecomposable Lorentzian symmetric spaces using the known classification of them.
More precisely, when the indecomposable Lorentzian symmetric space is not isotropic, its transvection group is solvable and the orbits of the isotropy group are parametrized by two variables. A limit formula is then obtained for the associated orbital integrals.
In summary, we get the following result.

\begin{theorem}\label{maintheorem}
On any connected simply connected indecomposable Lorentzian symmetric space, any compactly supported smooth function is determined in terms of its orbital integrals via a limit formula involving invariant differential operators.
\end{theorem}

Section \ref{sectionsymmetric} introduces the necessary background about symmetric spaces and gives the classification of indecomposable Lorentzian symmetric spaces due to M. Cahen and N. Wallach \cite{CaWa70}. 
Orbital integrals are defined in section \ref{sectionrankoneorbital} for semisimple pseudo-Riemannian symmetric spaces. We lay out S. Helgason's determination of a function in terms of its orbital integrals via a limit formula for even-dimensional isotropic Lorentzian symmetric spaces and J. Orloff's generalization which gives the keys to treat the odd-dimensional isotropic Lorentzian symmetric spaces.
The proof of these limit formulas uses some integral operators called Riesz potentials and the Laplace-Beltrami operator associated to the metric.
In section \ref{sectionsolvable}, we give an explicit description of the model spaces for solvable Lorentzian symmetric spaces. We compute the exponential mapping at any point and determine the orbits of the isotropy group.
Section \ref{sectionsolvableorbital} presents the steps and arguments leading to a limit formula to determine a function in terms of its orbital integrals on solvable Lorentzian symmetric spaces.


\section{Framework of symmetric spaces}\label{sectionsymmetric}

\subsection{General features of symmetric spaces}

We follow the description of symmetric spaces by O. Loos \cite{Loo69}.

\begin{definition}A \emph{symmetric space} is a smooth manifold $M$ endowed with a smooth map
$$s : M \times M \rightarrow M : (x, y) \mapsto s_x (y)$$
such that
\begin{enumerate}
\item $\forall x \in M$, $s_x$ is an involutive diffeomorphism,
\item $\forall x \in M$, $x$ is an isolated fixed point of $s_x$,
\item $\forall x, y$, $s_x \circ s_y \circ s_x = s_{s_x (y)}$.
\end{enumerate}
The diffeomorphism $s_x$ is called a \emph{symmetry} at $x$.\end{definition}

For any connected symmetric space $(M, s)$, we define the \emph{transvection group} $G (M,s)$ as the group generated by the automorphisms $s_x \circ s_y$ for any $x, y \in M$. There exists a Lie group structure on $G (M, s)$ such that it acts transitively on $M$. Therefore, $M$ is a homogeneous space.
\newline

Let us fix a base-point $x_0$ of $M$. We then consider the involutive automorphism of $G (M, s)$ defined as the conjugation by the symmetry at $x_0$. Then its differential at the neutral element of $G (M,s)$, denoted by $\sigma$ is an involutive automorphism of the Lie algebra $\mathfrak{g}$ of $G (M,s)$. It induces a decomposition
$$\mathfrak{g} = \mathfrak{k} \oplus \mathfrak{p}$$
into the eigenspaces $\mathfrak{k}$, $\mathfrak{p}$ of $\sigma$ with respect to the eigenvalues $1$, $-1$ respectively. In particular,
$$[\mathfrak{p}, \mathfrak{p}] = \mathfrak{k}, \quad [\mathfrak{p}, \mathfrak{k}] \subset \mathfrak{p}, \quad [\mathfrak{k}, \mathfrak{k}] \subset \mathfrak{k}.$$

\begin{definition}
A \emph{symmetric Lie algebra} is a finite dimensional real Lie algebra $\mathfrak{g}$ endowed with an involutive automorphisme $\sigma$ of $\mathfrak{g}$. This pair $(\mathfrak{g}, \sigma)$ is a \emph{transvection} symmetric Lie algebra if, in addition, $\mathfrak{k} = [\mathfrak{p}, \mathfrak{p}]$ where $\mathfrak{k}$, $\mathfrak{p}$ are the eigenspaces for $\sigma$ with respect to the eigenvalues $1$, $-1$ respectively.
\end{definition}

\begin{definition}
A transvection symmetric Lie algebra $(\mathfrak{g}, \sigma)$ is said to be \emph{effective} if one of the following equivalent conditions is satisfied
\begin{itemize}
\item $\mathfrak{k}$ contains no nonzero ideal of $\mathfrak{g}$,
\item $\mathfrak{k} \cap \mathfrak{z} (\mathfrak{g}) = \lbrace 0 \rbrace$,
\item the map $\mathrm{ad}_\mathfrak{g} ( \cdot ) \vert_\mathfrak{p} : \mathrm{k} \rightarrow \mathfrak{gl} (\mathfrak{p})$ is injective.
\end{itemize}
\end{definition}

The transvection symmetric Lie algebra $(\mathfrak{g}, \sigma)$ defined above for a connected symmetric space $(M, s)$, where $\mathfrak{g}$ is the Lie algebra of the transvection group $G (M)$ and $\sigma$ the differential of the map $(g \mapsto s_{x_0} \circ g \circ s_{x_0})$, is effective.
Conversely, for any symmetric Lie algebra $(\mathfrak{g}, \sigma)$, there exists a connected, simply connected symmetric space $(M,s)$ which is $G$-homogeneous for a Lie subgroup $G$ of the automorphisms of $(M,s)$ whose Lie algebra is $\mathfrak{g}$ and such that the differential of the induced involution is $\sigma$.
\newline

On the other hand, there exists a unique affine connection on $M$ such that every symmetry is an affine transformation.
This connection is complete, that is every geodesic is defined on $\R$. Therefore, at any $x \in M$, the exponential mapping, denoted by $\mathrm{Exp}_{x}$, is defined on the whole $T_x M$.
The well-known property about the exponential mapping of a symmetric space $(M,s)$ is 
$$\mathrm{Exp}_{x_0} (X) = \exp (\bar{X}).x_0, \quad \text{for } X \in T_{x_0} M$$
where $\exp : \mathfrak{g} \rightarrow G$ is the exponential mapping of the Lie group $G$ and $\bar{X}$ is the unique element in $\mathfrak{p}$ such that $\phi (\bar{X}) = X$ using the isomorphism
$$\phi : \begin{cases} \mathfrak{p} &\rightarrow T_{x_0} M\\
\bar{X} &\mapsto \frac{d}{dt} \big\vert_0 \exp (t \bar{X}).x_0 
\end{cases}.$$

In 1959 \cite{Hel59}, Sigurdur Helgason gave an explicit expression of the differential of the exponential mapping which is very useful for the computations.

\begin{lemma}\label{ExponentDiff}
Let $(M,s)$ be a symmetric space and $x_0$ be a base-point of $M$. If $\mathfrak{g} = \mathfrak{k} \oplus \mathfrak{p}$ is the decomposition of the transvection Lie algebra with respect to the natural involution, then the exponential mapping $\mathrm{Exp}_{x_0} : T_{x_0} M \rightarrow M$ associated to the canonical affine connection has differential
$$\big( \mathrm{Exp}_{x_0} \big)_{\ast X} = \tau (\exp (X))_{\ast x_0} \circ \sum_{k=0}^\infty \frac{\mathrm{ad} (X)\big\vert_\mathfrak{p}^{2k}}{(2k+1)!} , \quad \text{for } X \in T_{x_0} M \simeq \mathfrak{p},$$
where $\tau (\exp (X))$ is the diffeomorphism defined by the action of the group element $\exp (X)$ on $M$ and $\mathrm{ad}$ the adjoint action on the Lie algebra $\mathfrak{g}$.
\end{lemma}

\subsection{Rank of symmetric Lie algebras}

In \cite{LeMc76}, Lepowsky and McCollum defined the rank of any symmetric Lie algebra even when it is not necessarily semisimple. This requires introducing Cartan subspaces in a very general way.

\begin{definition}
Let $(\mathfrak{g}, \sigma)$ be a symmetric Lie algebra and let $\mathfrak{g} = \mathfrak{k} \oplus \mathfrak{p}$ be the usual decomposition of $\mathfrak{g}$ into the $(\pm 1)$-eigenspaces of $\sigma$. A \emph{Cartan subspace} of the symmetric Lie algebra is a subspace $\mathfrak{a}$ of $\mathfrak{p}$ such that
$$\mathfrak{a} = \mathfrak{p}_\mathfrak{a}^0 := \lbrace X \in \mathfrak{p} \mid \forall Y \in \mathfrak{a}, \, \exists n \geq 0, \, \mathrm{ad} (Y)^n (X) = 0 \rbrace.$$
\end{definition}

\begin{definition}
Let $(\mathfrak{g}, \sigma)$ be a symmetric Lie algebra and $\mathfrak{a}$ a subspace of $\mathfrak{p}$. Then $\mathfrak{a}$ is said to be a $\mathfrak{p}$-\emph{subalgebra} if
$$\forall X \in \mathfrak{a}, \quad \mathrm{ad}^2 (X) \mathfrak{a} \subset \mathfrak{a}.$$
Furthermore, it is said to be \emph{natural} if there exists $X_0 \in \mathfrak{a}$ such that $\mathrm{ad}^2 (X_0)$ induces a nonsingular endomorphism of $\mathfrak{p} / \mathfrak{a}$.
\end{definition}

\begin{proposition}[Lepowsky - McCollum, 1976 \cite{LeMc76}]
Let $(\mathfrak{g}, \sigma)$ be a symmetric Lie algebra and $\mathfrak{a}$ be a subspace of $\mathfrak{p}$. Then $\mathfrak{a}$ is a Cartan subspace if and only if $\mathfrak{a}$ is a minimal natural $\mathfrak{p}$-subalgebra of $\mathfrak{p}$.
\end{proposition}

\begin{theorem}[Lepowsky - McCollum, 1976 \cite{LeMc76}]
In any symmetric Lie algebra, there exist Cartan subspaces and they all have the same dimension. We call \emph{rank} of the symmetric Lie algebra this common dimension.
\end{theorem}

\subsection{Pseudo-Riemannian symmetric spaces}

We need to add a compatible metric on the symmetric spaces in order to define generalized spheres on them.

\begin{definition}
A \emph{pseudo-Riemannian symmetric space} is a symmetric space $(M, s)$ endowed with a pseudo-Riemannian metric $\hat{g}$ such that, for all $x \in M$, $s_x$ is an isometry of $(M, \hat{g})$.
\end{definition}

On a pseudo-Riemannian symmetric space, the unique affine connection such that every symmetry is an affine transformation is the Levi-Civita connection associated to the metric.

\begin{definition}\label{deftriple}
A \emph{pseudo-Riemannian symmetric Lie algebra} is a symmetric Lie algebra $(\mathfrak{g}, \sigma)$ endowed with a non-degenerate symmetric bilinear form $B$ on $\mathfrak{g}$ such that
\begin{enumerate}
\item $\forall X, Y \in \mathfrak{g}, \quad B (\sigma X, \sigma Y) = B (X, Y)$,
\item $\forall X, Y, Z \in \mathfrak{g}, \quad B (\mathrm{ad} (Z) X, Y) + B (X, \mathrm{ad} (Z) Y) = 0$.
\end{enumerate}
The \emph{signature} of the triple $(\mathfrak{g}, \sigma, B)$ is the signature of $B \vert_{\mathfrak{p} \times \mathfrak{p}}$ and its \emph{dimension} is the dimension of $\mathfrak{p}$.
\end{definition}

Note that the transvection pseudo-Riemannian symmetric Lie algebras are automatically effective.

\begin{definition}
An isomorphism of pseudo-Riemannian symmetric Lie algebras $(\mathfrak{g}_1, \sigma_1, B_1)$ and $(\mathfrak{g}_2, \sigma_2, B_2)$ is a Lie algebra isomorphism $\alpha : \mathfrak{g}_1 \rightarrow \mathfrak{g}_2$ such that $\alpha \circ \sigma_1 = \sigma_2 \circ \alpha$ and $\alpha^\ast B_2 = B_1$.
\end{definition}

\begin{theorem}
There is one-to-one correspondence between isometry classes of connected, simply connected pseudo-Riemannian symmetric spaces and isomorphism classes of transvection pseudo-Riemannian symmetric Lie algebras.
\end{theorem}

\begin{definition}
A connected pseudo-Riemannian manifold $(M, \hat{g})$ is said to be \emph{decomposable} if there exists a proper subspace $V$ of $T_{x_0} M$ which is invariant under the holonomy group $\mathrm{Hol} (M, x_0)$ at $x_0$ and such that $\hat{g}_{x_0} \vert_{V \times V}$ is non-degenerate. Otherwise $(M, \hat{g})$ is said to be \emph{indecomposable}.
\end{definition}

\begin{definition}
A pseudo-Riemannian symmetric Lie algebra $(\mathfrak{g}, \sigma, B)$ is said to be \emph{decomposable} if, when denoting by $\mathfrak{g} = \mathfrak{k} \oplus \mathfrak{p}$ the usual decomposition of $\mathfrak{g}$ with respect to $\sigma$, there exists a proper subspace $\mathfrak{q}$ of $\mathfrak{p}$ which is invariant under $\mathrm{ad} (\mathfrak{k})$ and such that $B \vert_{\mathfrak{q} \times \mathfrak{q}}$ is non-degenerate. Otherwise $(\mathfrak{g}, \sigma, B)$ is said to be \emph{indecomposable}.
\end{definition}

The de Rham-Wu theorem \cite{Wu64} asserts that any decomposable connected pseudo-Riemannian manifold which is simply connected and complete is isometric to a product of indecomposable ones. In addition, a pseudo-Riemannian product manifold is a symmetric space if and only if every factor of the product is a symmetric space. Finally, a connected, simply connected pseudo-Riemnnian symmetric space $(M, s, \hat{g})$ is indecomposable if and only if its associated transvection pseudo-Riemannian symmetric Lie algebra $(\mathfrak{g}, \sigma, B)$ is indecomposable.

\subsection{Classification of indecomposable Lorentzian symmetric spaces}

Indecomposable Lorentzian symmetric spaces fall into two categories : those whose transvection Lie algebra is semisimple and those whose transvection Lie algebra is solvable.

\begin{proposition}[M. Cahen and N. Wallach, 1970 \cite{CaWa70}]\label{theoremindecomplorentz}
Let $(\mathfrak{g}, \sigma, B)$ be an indecomposable  transvection Lorentzian symmetric Lie algebra. Then $\mathfrak{g}$ is either semisimple or solvable.
\end{proposition}

In the semisimple case, the associated Lorentzian symmetric spaces are automatically of constant sectional curvature as ensures the following theorem whose proof is based on Berger's list \cite{Ber57}.

\begin{theorem}[M. Cahen, J. Leroy, M. Parker, F. Tricerri and L. Vanhecke, 1990 \cite{Cah90}]
Let $(\mathfrak{g}, \sigma, B)$ be a semisimple indecomposable Lorentzian symmetric Lie algebra of dimension $\geq 3$. Then the sectional curvature of the associated Lorentzian symmetric space is constant and non-zero.
\end{theorem}

It is possible to prove this result without using Berger's list. It involves a result of A. J. Di Scala and C. Olmos \cite{ScaOlm01} about connected Lie subgroups of the Lorentzian group acting irreducibly on the flat Lorentzian vector space.

\begin{proof}
Let $\mathfrak{g} = \mathfrak{k} \oplus \mathfrak{p}$ be the usual decomposition of $\mathfrak{g}$ with respect to $\sigma$. Then the representation
$$\mathrm{ad}_\mathfrak{g} (\cdot) \vert_\mathfrak{p} : \mathfrak{k} \rightarrow \mathfrak{gl} (\mathfrak{p})$$
is faithful and irreducible. If $\dim (\mathfrak{p}) = n$, the image of $\mathfrak{k}$ by $\mathrm{ad}$ is contained in $\mathfrak{o} (\mathfrak{p}, B \vert_\mathfrak{p}) \simeq \mathfrak{o} (1, n-1)$.
By integrating, we get a connected Lie subgroup $K$ of $SO (1, n-1)$ which acts irreducibly on $(\R^{1,n-1}, \mathbb{I}_{1, n-1})$. By theorem 1.1 in \cite{ScaOlm01}, $K = SO_0 (1, n-1)$. Therefore, the sectional curvature of the associated Lorentzian symmetric space is constant and non-zero.
\end{proof}

In dimension $2$, the Lorentzian symmetric spaces are automatically of constant sectional curvature. The classification in the semisimple case is given by the following theorem.

\begin{theorem}[S. Helgason, 1959 \cite{Hel59}]\label{theoremLorentz}
Any Lorentzian symmetric space with non-zero constant sectional curvature is locally isometric to one of the two following model spaces, up to a positive constant factor on the metric,
\begin{enumerate}
\item $Q_{(+1)} := \lbrace (x_1, ..., x_{n+1}) \in \R^{n+1} \mid x_1^2 - x_2^2 - ... - x_n^2 + x_{n+1}^2 = 1 \rbrace$,
\item $Q_{(-1)} := \lbrace (x_1, ..., x_{n+1}) \in \R^{n+1} \mid x_1^2 - x_2^2 - ... - x_n^2 - x_{n+1}^2 = -1 \rbrace$,
\end{enumerate}
endowed with the Lorentzian metric induced by the flat metric on $\R^{1, n}$.
\end{theorem}

The proof of this theorem, found in \cite{Hel59}, relies on the fact that a pseudo-Riemannian symmetric space is determined on a neighborhood of any point by its metric and curvature tensor at this point. Finally, it is clear that any Lorentzian symmetric space whose sectional curvature is constant and non-zero is of rank one.
\newline

In the solvable case, the classification is given at the algebraic level. Any solvable indecomposable Lorentzian symmetric Lie algebra is completely described in a particular basis by $n-2$ nonzero real numbers.

\begin{theorem}[M. Cahen and N. Wallach, 1970 \cite{CaWa70}]\label{solvabletriple}
Let $(\mathfrak{g}, \sigma, B)$ be a solvable indecomposable Lorentzian symmetric Lie algebra and $\mathfrak{g} = \mathfrak{k} \oplus \mathfrak{p}$ be the decomposition associated to the involution $\sigma$. There exist $\lambda_1, ..., \lambda_{n-2} \in \R_0$ and a basis $$\lbrace Z, U, W_1, ..., W_{n-2}, K_1, ..., K_{n-2} \rbrace$$ of $\mathfrak{g}$ such that
\begin{itemize}
\item $\mathfrak{k} = \R K_1 \oplus ... \oplus \R K_{n-2}$, $\mathfrak{p} = \R Z \oplus \R U \oplus \R W_1 \oplus ... \oplus \R W_{n-2}$,
\item $\mathfrak{z} (\mathfrak{g}) = \R Z$, $[U, K_i] = \lambda_i W_i$, $[U, W_i] = - K_i$, $[W_j, K_i] = \lambda_i \delta_{ij} Z$, $[W_i, W_j] = 0 = [K_i, K_j]$,
\item $B (Z, Z) = 0 = B (U, U)$, $B(Z, U) = 1$, $B (W_i, W_j) = - \delta_{ij}$ and $B (K_i, K_j) = - \lambda_i \delta_{ij}$.
\end{itemize}
\end{theorem}

\begin{proposition}
Let $(\mathfrak{g}, \sigma, B)$ be a solvable indecomposable Lorentzian symmetric Lie algebra,  explicitly described in theorem \ref{solvabletriple} for a certain choice of parameters $\lambda_1, ..., \lambda_{n-2} \in \R_0$. 
The rank of this symmetric Lie algebra is equal to $2$ and any Cartan subspace is of the form
$$\R Z \oplus \R \Big(U - \sum_{i=1}^{n-2} \lambda_i y_i W_i \Big)$$
for some fixed $y_1, ..., y_{n-2} \in \R$.
\end{proposition}

\begin{proof}
If $\mathfrak{a}$ is a Cartan subspace of $\mathfrak{p}$, then $\mathfrak{z} (\mathfrak{g}) \cap \mathfrak{p} \subset \mathfrak{a}$. Otherwise $\mathfrak{a}$ wouldn't be natural. Therefore, $Z \in \mathfrak{a}$. Nonetheless, $\R Z \neq \mathfrak{a}$ because $\mathrm{ad}^2 (Z)$ doesn't induce a nonsingular endomorphism of $\mathfrak{p} / \mathfrak{a}$. Thus $M$ is not of rank one.

Let us consider $\mathfrak{a}_0 = \R Z \oplus \R U$. It is clearly a $\mathfrak{p}$-subalgebra. Moreover, it is natural since $\mathrm{ad}^2 (U)$ induces a nonsingular endomorphism of $\mathfrak{p} / \mathfrak{a}$. At last, the minimality is obvious. As a conclusion, $\mathfrak{a}_0$ is a Cartan subspace of $\mathfrak{p}$ and $(\mathfrak{g}, \sigma, B)$ is of rank $2$. 

Back to the beginning of the proof, if $\mathfrak{a}$ is any Cartan subspace of $\mathfrak{p}$, we know that it is of dimension $2$ and it contains $Z$.
Then it is of the form
$$\mathfrak{a} = \R Z \oplus \R \Big(\nu U + \sum_{i=1}^{n-2} \mu_i W_i \Big)$$
for some $\nu, \mu_1, ...,\mu_{n-2} \in \R$ not all zero.
Since $\mathrm{ad} (\sum_{i=1}^{n-2} \mu_i W_i)^2$ doesn't induce a nonsingular endomorphism of $\mathfrak{p} / \mathfrak{a}$, then $\nu \neq 0$ hence the result.
\end{proof}

\subsection{Isotropic pseudo-Riemannian symmetric spaces}

On any connected pseudo-Riemannian symmetric space $(M, s, \hat{g})$, the group of isometries, denoted by $I (M, \hat{g})$, acts transitively.
Then the isotropy group $K_x$ of any point $x \in M$ acts on the tangent space $T_x M$ in the following way
$$K_x \times T_xM \rightarrow T_xM : (k, X) \mapsto \tau (k)_{\ast x} (X)$$
where $\tau (k)$ is the action of $k$ on $M$.
The second definition below comes from A. J. Wolf's paper \cite{Wol67} and the first one is a weaker version of it.

\begin{definition}
Let $(M, s, \hat{g})$ be a connected pseudo-Riemannian symmetric space. Then $M$ is said to be \emph{quasi-isotropic} if, for any $x \in M$, the pseudo-spheres in $T_x M$, namely
$$\Sigma_{\alpha} (x) := \lbrace X \in T_x M \mid \hat{g}_x (X, X) = \alpha \rbrace$$
for any $\alpha \in \R_0$, are orbits under the action of the isotropy group $K_x$ of $x$ in $I (M, \hat{g})$, the group of isometries.

Furthermore, $M$ is said to be \emph{isotropic} if it is quasi-isotropic and, for any $x \in M$, the light cone in $T_xM$, namely
$$\Sigma_0 (x) := \lbrace X \in T_xM \mid \hat{g}_x (X, X) = 0, \quad X \ne 0 \rbrace,$$
is also an orbit of the isotropy group $K_x$ of $x$ in $I (M, \hat{g})$.
\end{definition}

\begin{proposition}[J. Orloff, 1987 \cite{Orl87}]\label{propisotropicandrank}
Let $(M, s, \hat{g})$ be a connected pseudo-Riemannian symmetric space whose group of isometries is semisimple. Then $M$ is quasi-isotropic if and only if it is of rank one.
\end{proposition}

Note that there exist rank-one pseudo-Riemannian symmetric spaces whose group of isometries is semisimple which are not isotropic like the space $X = \mathrm{SL} (n,\R) / \mathrm{GL} (n-1,\R)$ studied by M. T. Kosters and G. Van Dijk in \cite{KosDij86}.


\section{Limit formulas on Lorentzian symmetric spaces of constant sectional curvature}\label{sectionrankoneorbital}

We address the problem of determining a function in terms of its orbital integrals on Lorentzian symmetric spaces of constant sectional curvature.
These are either the flat Lorentzian vector space or, up to a positive constant factor on the metric, one of the two model spaces in theorem \ref{theoremLorentz}.
S. Helgason solved the problem when their dimension is even via a limit formula \cite{Hel59}. 
J. Orloff extended it to rank-one semisimple pseudo-Riemannian symmetric spaces \cite{Orl87} and gave the keys to treat the odd-dimensional Lorentzian symmetric spaces of constant sectional curvature.

\begin{definition}
Let $(M, s, \hat{g})$ be a pseudo-Riemannian symmetric space. Then $M$ is said to be \emph{semisimple} if there exists a Lie subgroup $G$ of $I (M, \hat{g})$ which is semisimple, acts transitively on $M$ and is invariant under the conjugation by $s_{x_0}$ where $x_0$ is a base-point of $M$.
\end{definition}


\subsection{Definition of orbital integrals}\label{subsectiondeforbint}

On the pseudo-Euclidean vector space, that is the space $\R^n$ endowed with the flat metric of signature $(p,q)$, the orbital integrals are defined as pseudo-spherical integrals.

\begin{definition}
Let $\langle ., . \rangle$ be the standard inner product of signature $(p, q)$ on $\R^n$. For any function $f \in \mathcal{C}_c (\R^n)$, the \emph{orbital integrals} of $f$ are the pseudo-radial functions denoted by $(M_+ f)$ and $(M_- f)$ and defined by
$$(M_\pm f) (r) := \frac{1}{r^{n-1}} \int_{\Sigma_{\pm r^2}} f(x)\, d\eta (x), \quad \text{for } r > 0,$$
where $d\eta$ is the measure induced by the metric on the pseudo-spheres centered at $0$ in $\R^n$, namely
$$\Sigma_{r^2} := \lbrace x \in \R^n \mid \langle x, x \rangle = r^2\rbrace,\quad \Sigma_{-r^2} := \lbrace x \in \R^n \mid \langle x, x \rangle = -r^2\rbrace.$$
\end{definition}

Let $(M, s, \hat{g})$ be a semisimple pseudo-Riemannian symmetric space whose metric is of signature $(p,q)$ and $x_0$ a base-point of $M$.
Let $G$ be a semisimple Lie subgroup of $I (M, \hat{g})$ which acts transitively on $M$ and is invariant under the conjugation by $s_{x_0}$.
Then the Lie algebra $\mathfrak{g}$ of $G$ decomposes into the $(\pm 1)$-eigenspaces of the involution given by the conjugation by $s_{x_0}$
$$\mathfrak{g} = \mathfrak{k} \oplus \mathfrak{p}$$
and we get the usual isomorphism $\phi : \mathfrak{p} \rightarrow T_{x_0} M$ defined in section \ref{sectionsymmetric}.
Following J. Orloff's paper \cite{Orl87}, we assume that the isotropy group $K$ of $x_0$ is connected and we define the orbital integrals as follows.

\begin{definition}
Let $\mathfrak{a}$ be a Cartan subspace of $\mathfrak{g}$.
For any function $f \in \mathcal{C}_c (M)$ and any point $x \in M$, the \emph{orbital integrals} of $f$ at $x$ are given by
$$(M^X f) (x) := \int_{K/H_X} f (gk.\mathrm{Exp}_{x_0} (X))\; d\mu (kH_X),$$
where $g \in G$ such that $x = g.x_0$, $H_X$ is the stabilizer of $X$ in $K$ and $d\mu$ is the $K$-invariant measure on $K/H_X$ induced by the metric on the $K$-orbit of $X \in T_{x_0} M$ providing that
$$[\bar{X}, \mathfrak{k}] = \mathfrak{a}^\perp$$
where $\bar{X} \in \mathfrak{p}$ such that $\phi (\bar{X}) = X$ and $\mathfrak{a}^\perp$ is the orthogonal complement of $\mathfrak{a}$ in $\mathfrak{p}$ with respect to the metric $\phi^\ast \hat{g}_{x_0}$.
\end{definition}

\subsection{Limit formulas in the even-dimensional Lorentzian case}\label{subsectionevendimlorentz}

Let $(M, s, \hat{g})$ be a Lorentzian symmetric space of non-zero constant sectional curvature.
On the model spaces introduced in theorem \ref{theoremLorentz}, the group of isometries is respectively $O (2, n-1)$ if $M = Q_{(+1)}$ and $O (1, n)$ if $M = Q_{(-1)}$. Therefore,
$$Q_{(+1)} \simeq SO_0 (2, n-1) / SO_0 (1,n-1), \quad Q_{(-1)} \simeq SO_0 (1,n) / SO_0 (1,n-1),$$
with the base-point $x_0 = (0, ..., 0, 1)$, and the orbit of any vector $X \in T_{x_0}M$ such that $\hat{g}_{x_0} (X, X) \neq 0$ under the action of the isotropy group $K = SO_0 (1,n-1)$ is the connected component of a pseudo-sphere.
We then use the following notation for the orbital integrals
\begin{align*}
&(M^r_+ f) (x) := (M^X f) (x), \quad \text{where } X = (r, 0, ..., 0) \in \R^{1,n-1} \simeq T_{x_0} M,\\
&(\tilde{M}^r_+ f) (x) := (M^X f) (x), \quad \text{where } X = (-r, 0, ..., 0) \in \R^{1,n-1} \simeq T_{x_0} M,\\
&(M^r_- f) (x) := (M^X f) (x), \quad \text{where } X = (0, ..., 0, r) \in \R^{1,n-1} \simeq T_{x_0} M,\\
&(\tilde{M}^r_- f) (x) := (M^X f) (x), \quad \text{where } X = (0, ..., 0, -r) \in \R^{1,n-1} \simeq T_{x_0} M,
\end{align*}
for any $r > 0$. Note that $M^r_- f = \tilde{M}^r_- f$ except when $n = \dim (M) = 2$.
\newline

As in S. Helgason's book \cite{Hel10}, we focus on the model space $M = Q_{(-1)}$ and the orbital integrals $M_+^r f$ to exhibit results and arguments leading to the limit formula. Everything works in the same manner with the other model space and the other series of orbital integrals.

\begin{proposition}[S. Helgason, 1959 \cite{Hel59}]\label{propDAlembertOrbitalint}
Let $\square$ be the Laplace-Beltrami operator associated to the Lorentzian metric $\hat{g}$ on $Q_{(-1)}$.
For any function $f \in \mathcal{C}^\infty_c (Q_{(-1)})$, any point $x \in M$ and any pseudo-radius $0 < r < r_0$,
\begin{align*}
\square (M^r_+ f) (x) &= M^r_+ ( \square f) (x)\\
&= \frac{1}{A(r)} \frac{\partial}{\partial r} \Big( A(r) \frac{\partial}{\partial r} (M^r_+ f) (x) \Big).
\end{align*}
where $A(r) := \sinh^{n-1} (r)$.
\end{proposition}

\begin{definition} For any function $f \in \mathcal{C}_c (Q_{(-1)})$, its \emph{Riesz potentials} at any point $x \in Q_{(-1)}$ and for any parameter $\lambda \in \mathbb{C}$ such that $\mathrm{Re} (\lambda) > n$, are given by
$$(I^\lambda_+ f) (x) := \frac{1}{H_n (\lambda)} \int_{D^+_x} f(y) \sinh^{\lambda - n} \big(\sqrt{\hat{g}_x (Y, Y)}\big)\, dm(y), \quad y = \mathrm{Exp}_x (Y),$$
where $dm$ is the measure induced by the metric $\hat{g}$ on $Q_{(-1)}$,
$D^+_x$ is the connected component of
$\lbrace \mathrm{Exp}_x (Y) \mid Y \in T_x Q_{(-1)},\; \hat{g}_x (Y, Y) > 0\rbrace$
containing the vector $g.\mathrm{Exp}_{x_0}(1, 0, ..., 0)$ if $x = g.x_0$
and $$H_n (\lambda) := 2^{\lambda - 1} \pi^{\frac{n-2}{2}} \Gamma \Big(\frac{\lambda}{2}\Big) \Gamma \Big(\frac{\lambda + 2 - n}{2}\Big).$$
\end{definition}

The Riesz potentials associate to a function $f$ a one-parameter family of integrals $(I^\lambda f )$ defined for $\lambda \in \mathbb{C}$ such that $\text{Re} (\lambda) > n$.

\begin{proposition}[S. Helgason, 1959 \cite{Hel59}]\label{Lorentzprop}
For any function $f \in \mathcal{C}^\infty_c (Q_{(-1)})$, any point $x \in Q_{(-1)}$ and any parameter $\lambda \in \mathbb{C}$ such that $\text{Re} (\lambda) > n$,
\begin{itemize}
\item[(i)]$\square (I^\lambda_+ f) (x) = I^\lambda_+ (\square f) (x)$,
\item[(ii)] if $\text{Re} (\lambda) > n+2$, $\square (I^\lambda_+ f) (x) = (\lambda - n) (\lambda - 1) (I^\lambda_+ f) (x) + (I^{\lambda -2}_+ f) (x)$,
\item[(iii)] $(I^\lambda_+ f) (x)$ extends holomorphically to $\mathbb{C}$ in the $\lambda$-variable so that the value at $\lambda = 0$ is
$$(I^0_+ f) (x) = f(x).$$
\end{itemize}
\end{proposition}


Finally, the Riesz potentials express in terms of the orbital integrals in the following way 
\begin{equation}\label{Lorentzlink}
(I^\lambda_+ f) (x) = \frac{1}{H_n (\lambda)} \int_0^\infty (M^r_+ f) (x) \sinh^{\lambda - 1} (r) dr.
\end{equation}
This yields the desired limit formula whose proof is sketched below.

\begin{theorem}[S. Helgason, 1959 \cite{Hel59}]\label{limitformevendim}
We assume $n := \text{dim} (Q_{(-1)}) > 2$ to be even. Then there exists a polynomial $P$ and a real number $c$ such that, for any $f \in \mathcal{C}^\infty_c (Q_{(-1)})$,
$$f(x) = c. \underset{r \rightarrow 0^+}{\lim}\; r^{n-2} P (\square) (M^r_+ f) (x),$$
where $\square$ is the associated Laplace-Beltrami operator on $Q_{(-1)}$.
\end{theorem}

\begin{proof}
First, when the dimension $n$ is strictly greater than $2$, the limit
$$\underset{r \rightarrow 0^+}{\lim}\; r^{n-2} (M^r_+ f) (x) = \underset{r \rightarrow 0^+}{\lim}\; \sinh^{n-2} (r) (M^r_+ f) (x)$$
exists. Moreover, thanks to formula \ref{Lorentzlink}, the Riesz potential $(I^\lambda_+ f) (x)$ is equal to the following Riemann-Liouville integral
$$\frac{1}{\Gamma (\mu)} \int_0^\infty F(r) \sinh^{\mu-1} (r) dr,$$
where $F(r) := \sinh^{n-2} (r) (M^r_+ f) (x)$ and $\mu := \lambda - n + 2$. We thus get
$$(I^{n-2}_+ f) (x) =  \frac{(4\pi)^{(2-n)/2}}{\Gamma ((n-2)/2)} \underset{r \rightarrow 0^+}{\lim}\; F(r) = \frac{(4\pi)^{(2-n)/2}}{\Gamma ((n-2)/2)} \underset{r \rightarrow 0^+}{\lim}\; r^{n-2} (M^r_+ f) (x).$$

Since $n$ is also even, we deduce from proposition \ref{Lorentzprop} the existence of a polynomial $P$ such that
$$P(\square) (I^{n-2}_+ f) (x) = I^{n-2}_+ (P(\square) f) (x) = f(x).$$
Therefore,
$$f(x) = \frac{(4\pi)^{(n-2)/2}}{\Gamma ((n-2)/2)} \underset{r \rightarrow 0^+}{\lim}\; r^{n-2} M^r_+ (P (\square) f)) (x)$$
and proposition \ref{propDAlembertOrbitalint} leads to the limit formula.
\end{proof}

\subsection{Generalization to rank-one semisimple symmetric spaces}

Let $(M, s, \hat{g})$ be a semisimple pseudo-Riemannian symmetric space whose metric is of signature $(p,q)$ and $x_0$ a base-point of $M$.
We assume that $M$ is of rank one. Then by proposition \ref{propisotropicandrank}, $M$ is quasi-isotropic.
Like in the Lorentzian case in subsection \ref{subsectionevendimlorentz}, the orbital integrals of a function $f \in \mathcal{C}_c (M)$, namely $M^X f$, for vectors $X \in T_{x_0}M$ such that $\hat{g}_{x_0} (X, X)$ are then integrals over connected components of the pseudo-spheres in $M$. We also adopt the following alternative notation
\begin{align*}
&(M^r_+ f) (x) := (M^X f) (x), \quad \text{where } X = (r, 0, ..., 0) \in \R^{p,q} \simeq T_{x_0} M,\\
&(M^r_- f) (x) := (M^X f) (x), \quad \text{where } X = (0, ..., 0, r) \in \R^{p,q} \simeq T_{x_0} M,
\end{align*}
for any $r > 0$.

\begin{remark}
There exists $r_0 > 0$ such that the exponential mapping of $M$ at $x_0$ is a diffeomorphism from the open subset $$\lbrace X \in T_{x_0} M \mid -r^2_0 < \hat{g}_{x_0} (X, X) < r^2_0 \rbrace$$ of $T_{x_0}M$ to its image.
Therefore, the orbital integrals on $M$ at the base-point $x_0$ are the pseudo-spherical integrals on the pseudo-Euclidean vector space through the exponential mapping except in the Lorentzian case where $\Sigma_{r^2}$ is not connected. More precisely, whenever $p > 1$ and $q > 1$, for any function $f \in \mathcal{C}_c (M)$ and any $0 < r < r_0$,
$$(M^r_\pm f) (x_0) = M_\pm (f \circ \mathrm{Exp}_{x_0}) (r),$$
where we identify $(T_{x_0}M, \hat{g}_{x_0})$ with $(\R^n, \langle .,. \rangle_{p,q})$.
\end{remark}

In order to get a limit formula on $M$, J. Orloff solved the problem of determining a function in terms of its orbital integrals on the flat space $\R^n$ endowed with the standard inner product of signature $(p,q)$. He studied the generalized Riesz potentials defined below.

\begin{theorem}[J. Orloff, 1987 \cite{Orl87}]\label{flatRieszpot}
Let $D_+ := \lbrace x \in \R^n \mid \langle x, x \rangle > 0\rbrace$ and $D_- := \lbrace x \in \R^n \mid \langle x, x \rangle < 0 \rbrace$. For any function $f \in \mathcal{S} (\R^n)$, we define its \emph{Riesz potentials} for any parameter $\lambda \in \C$ such that $\mathrm{Re} (\lambda) > n$
\begin{align*}
&I^\lambda_+ f := \frac{1}{H_n (\lambda)} \int_{D_+} f(x) \vert \langle x, x \rangle \vert^{\frac{\lambda - n}{2}} dx,\\
&I^\lambda_- f := \frac{1}{H_n (\lambda)} \int_{D_-} f(x) \vert \langle x, x \rangle \vert^{\frac{\lambda - n}{2}} dx,\\
&I^\lambda_0 f := \Gamma \Big(\frac{\lambda - n + 2}{2} \Big) \Big( I^\lambda_+ f - \cos \Big( \frac{\lambda - n + 2}{2} \pi \Big) I^\lambda_- f \Big),
\end{align*}
where $H_n (\lambda) := 2^{\lambda - 1} \pi^{\frac{n-2}{2}} \Gamma \big( \frac{\lambda}{2}\big) \Gamma \big( \frac{\lambda - n + 2}{2}\big)$. Then
\begin{enumerate}
\item $I^\lambda_+ f$, $I^\lambda_- f$ and $I^\lambda_0 f$ extend to entire functions in $\lambda$,
\item $I^{\lambda+2}_+ (Lf) = I^\lambda_+ f$ and $I^{\lambda+2}_- (Lf) = - I^\lambda_- f$ where $L$ is the Laplace-Beltrami operator associated to $\langle ., . \rangle$,
\item $I^0_+ f = 2 \sin (p\frac{\pi}{2}) f(0)$ and $I^0_- f = 2 \sin (q \frac{\pi}{2}) f(0),$
\item The map $(f \mapsto I^\lambda_{\pm} f)$ is a tempered distribution for all $\lambda$,
\item For $p$ and $q$ both even, $I^0_0 f = \frac{(-1)^{q/2} 2\pi}{(\frac{n-2}{2})!} f(0)$.
\end{enumerate}
\end{theorem}

\begin{theorem}[J. Orloff, 1987 \cite{Orl87}]
Whenever $p > 1$ and $q > 1$, for any function $f \in \mathcal{C}^\infty_c (\R^n)$, we have the following limit formulas.
\begin{enumerate}
\item If $p$ and $q$ are both odd then
$$\underset{r \rightarrow 0^+}{\lim} r^{n-2} \Big( \frac{d^2}{dr^2} + \frac{n-1}{r} \frac{d}{dr} \Big)^{\frac{n-2}{2}} (M_+ f) (r) = c f(0)$$
and $$\underset{r \rightarrow 0^+}{\lim} r^{n-2} \Big( -\frac{d^2}{dr^2} - \frac{n-1}{r} \frac{d}{dr} \Big)^{\frac{n-2}{2}} (M_- f) (r) = c f(0),$$
\item If $p$ is odd and $q$ is even then
$$\underset{r \rightarrow 0^+}{\lim} \Big( \frac{d}{dr} \Big)^{n-2} r^{n-2} (M_+ f) (r) = c f(0),$$
\item If $p$ is even and $q$ is odd then
$$\underset{r \rightarrow 0^+}{\lim} \Big( \frac{d}{dr} \Big)^{n-2} r^{n-2} (M_- f) (r) = c f(0),$$
\item If $p$ and $q$ are both even then
\end{enumerate}
$$\underset{r \rightarrow 0^+}{\lim} \Big( \frac{d}{dr} \Big)^{n-2} r^{n-2} (M_+ f) (r) + (-1)^\frac{n}{2} \underset{r \rightarrow 0^+}{\lim} \Big( \frac{d}{dr} \Big)^{n-2} r^{n-2} (M_- f) (r) = c f(0),$$
where, in each equation, $c$ is a nonzero constant independent of $f$.
\end{theorem}

Then the Riesz potentials and the limit formulas are lifted to the symmetric space $M$ via its exponential mapping.

\begin{theorem}[J. Orloff, 1987 \cite{Orl87}]
For $p > 1$ and $q > 1$, and $f \in \mathcal{C}^\infty_c (M)$,
\begin{enumerate}
\item If $p$ and $q$ are both odd,
$$\underset{r \rightarrow 0^+}{\lim} \sinh^{n-2} (r) P_+ (L^+_r) (M^r_+ f) (x_0) = c f(x_0),$$
and $$\underset{r \rightarrow 0^+}{\lim} \sin^{n-2} (r) P_- (L^-_r) (M^r_- f) (x_0) = c f(x_0),$$
\item If $p$ is odd and $q$ is even,
$$\underset{r \rightarrow 0^+}{\lim} \Big( \frac{d}{dr} \Big)^{n-2} r^{n-2} (M^r_+ f) (x_0) = c f(x_0),$$
\item If $p$ is even and $q$ is odd,
$$\underset{r \rightarrow 0^+}{\lim} \Big( \frac{d}{dr} \Big)^{n-2} r^{n-2} (M^r_- f) (x_0) = c f(x_0),$$
\item If $p$ and $q$ are both even,
\end{enumerate}
$$\underset{r \rightarrow 0^+}{\lim} \Big( \frac{d}{dr} \Big)^{n-2} r^{n-2} (M^r_+ f) (x_0) + (-1)^\frac{n}{2} \underset{r \rightarrow 0^+}{\lim} \Big( \frac{d}{dr} \Big)^{n-2} r^{n-2} (M^r_+ f) (x_0) = c f(x_0),$$
where $L_r^+$ (respectively $L_r^-$) is the radial part of the Laplace-Beltrami operator in polar geodesic coordinates on $D_+$ (respectively $D_-$), $P_+$ and $P_-$ are polynomials and, in each equation, $c$ is a non-zero constant independent of $f$.
\end{theorem}

\subsection{Summarizing result in the Lorentzian case}

For Lorentzian symmetric spaces of constant sectional curvature, we collect the limit formulas given by S. Helgason in the even-dimensional case and we use the arguments of J. Orloff to get the limit formulas in odd-dimensional case. This gives the following summarizing result.

\begin{theorem}[S. Helgason, 1959 \cite{Hel59} - J. Orloff, 1987 \cite{Orl87}]
Let $(M, s, \hat{g})$ be either the flat Lorentzian vector space $\R^{1,n-1}$ or one of the two model spaces $Q_{(+1)}$ and $Q_{-1}$ introduced in theorem \ref{theoremLorentz}.
Let $\kappa$ be the constant sectional curvature of $M$. If $n = \dim (M)$, for any $f \in \mathcal{C}_c^\infty (M)$,
\begin{enumerate}
\item If $n > 2$ is even,
$$f (x) = c \underset{r \rightarrow 0^+}{\lim} \; r^{n-2} P^\kappa_+ (\square) (M^r_+ f) (x)$$
and
$$f (x) = c \underset{r \rightarrow 0^+}{\lim} \; r^{n-2} P^\kappa_- (\square) (M^r_- f) (x),$$
\item If $n = 2$,
$$f(x) = -\frac{1}{2}\; \underset{r \rightarrow 0^+}{\lim} \; r \frac{d}{dr} (M^r_+ f) (x)$$
and
$$f(x) = -\frac{1}{2}\; \underset{r \rightarrow 0^+}{\lim} \; r \frac{d}{dr} (M^r_- f) (x),$$
\item If $n$ is odd,
$$f(x) = c \underset{r \rightarrow 0^+}{\lim} \Big( \frac{d}{dr} \Big)^{n-2} \big(r^{n-2} (M^r_+ f) (x)\big),$$
\end{enumerate}
where $P^\kappa_+$ and $P^\kappa_-$ are polynomials and, in each equation, $c$ is non-zero constant independent of $f$.
\end{theorem}

In order to prove theorem \ref{maintheorem}, taking theorem \ref{theoremindecomplorentz} into account, it remains to look at orbital integrals on solvable indecomposable Lorentzian symmetric spaces.


\section{Solvable indecomposable Lorentzian symmetric spaces}\label{sectionsolvable}

Let $(\mathfrak{g}, \sigma, B)$ be one of the triples described in theorem \ref{solvabletriple} for a particular choice of parameters $\lambda_1, ..., \lambda_{n-2} \in \R_0$.
Up to isomorphism, the connected, simply connected Lie group associated to $\mathfrak{g}$ is the smooth manifold
$$G := \R^{2n-2}$$
endowed with the group law
$$(t, p, q, r) \ast (t', p',q', r') := (t + t', (p,q,r) \cdot \varphi_t (p',q',r') )$$
for $t,t',r,r' \in \R$, $p,p',q,q' \in \R^{n-2}$, where $\cdot$ is the \emph{weighted} Heisenberg product given by
$$(p, q, r) \cdot (p', q', r') := (p+p', q+q', r+r' + \frac{1}{2} \sum_{i=1}^{n-2} \lambda_i (q_i p'_i - q'_i p_i))$$
and
\begin{align*}
\varphi_t (p',q',r') := &\Big(\ch (t \sqrt{\vert \lambda_i \vert}) p'_i - \frac{1}{\sqrt{\vert \lambda_i \vert}} \sh (t \sqrt{\vert \lambda_i \vert}) q'_i,\\
& \frac{\lambda_i}{\sqrt{\vert \lambda_i \vert}} \sh (t \sqrt{\vert \lambda_i \vert}) p'_i + \ch (t \sqrt{\vert \lambda_i \vert}) q'_i {\;}_{(1 \leq i \leq n-2)}, r' \Big).
\end{align*}
In the above formula and in the following, the notation $\sh (t \sqrt{\vert \lambda_i \vert})$ (respectively $\ch (t \sqrt{\vert \lambda_i \vert})$) means $\sin (t \sqrt{\lambda_i})$ (respectively $\cos (t \sqrt{\lambda_i})$) when $\lambda_i > 0$ and $\sinh (t \sqrt{-\lambda_i})$ (respectively $\cosh (t \sqrt{-\lambda_i})$) when $\lambda_i < 0$.

Therefore, $G$ is a semi-direct product of $\R$ and a \emph{weighted} Heisenberg group.
The connected Lie subgroup of $G$ whose Lie algebra is $\mathfrak{k}$, namely the isotropy group, is given by 
$$K := \lbrace (t, p, q, r) \in G \mid t = q_1 = ... = q_{n-2} = r = 0 \rbrace.$$
Let us define $\pi : G \rightarrow \R^n$ such that
\begin{align*}
\pi (t,p,q,r) &:= \Big(t, \frac{\lambda_i}{\sqrt{\vert \lambda_i \vert}} \sh (t \sqrt{\vert \lambda_i \vert}) p_i - \ch (t \sqrt{\vert \lambda_i \vert}) q_i {\;}_{(1 \leq i \leq n-2)},\\
&\quad \quad r + \sum_{i=1}^{n-2} \frac{\lambda_i}{2} \Big( \frac{\lambda_i}{\sqrt{\vert \lambda_i \vert}} \sh (t \sqrt{\vert \lambda_i \vert}) p_i - \ch (t \sqrt{\vert \lambda_i \vert}) q_i\Big)\\
&\qquad \qquad \qquad \Big( p_i \ch (t \sqrt{\vert \lambda_i \vert}) + \frac{q_i}{\sqrt{\vert \lambda_i \vert}} \sh (t \sqrt{\vert \lambda_i \vert})\Big) \Big)
\end{align*}
which factors the natural projection from $G$ onto $G/K$ and yields an isomorphism between $M := G/K$ and $\R^n$. We use this identification in the sequel.

The unique involutive automorphism of $G$ whose differential is $\sigma$ writes
$$\tilde{\sigma} : G \rightarrow G : (t,p,q,r) \mapsto (-t, p, -q, -r).$$
The associated symmetric structure on $M$ is given by
$$s_{\pi (g)} (\pi (g')) := \pi (g \sigma (g^{-1} g')), \quad \text{for } g, g' \in G.$$
If we denote by $(t, x_1, ..., x_{n-2}, v)$ the global coordinates on $M$ identified with $\R^n$ as above, the metric on $M$ associated to the bilinear form $B$ is given by
$$\hat{g} := \Big( \sum_{i=1}^{n-2} \lambda_i x_i^2 \Big) dt \otimes dt + dt\otimes dv + dv \otimes dt - \sum_{i=1}^{n-2} dx_i \otimes dx_i.$$

As a conclusion, up to isometry, the connected, simply connected, indecomposable solvable Lorentzian symmetric space associated to the triple $(\mathfrak{g}, \sigma, B)$ is $(M, s, \hat{g})$.
It is called a \emph{Cahen-Wallach space}. In dimension $4$, it is an example of pp-waves in the Brinkmann class which are idealized gravitational wave models in general relativity\cite{Bri25, EhlKun62, CosFloHer16}.


\subsection{The exponential mapping}

\begin{proposition}
Let $(M,s,\hat{g})$ be the Cahen-Wallach space corresponding to parameters $\lambda_1, ..., \lambda_{n-2} \in \R_0$.
The exponential mapping at any point $y = (t_0, x_0, v_0) \in M$ is given by
\begin{align*}
\mathrm{Exp}_{y} &(\bar{b}) = \Big( b_0 + t_0, \frac{\sh (b_0 \sqrt{\vert \lambda_i \vert})}{b_0 \sqrt{\vert \lambda_i \vert}} b_i + \ch (b_0 \sqrt{\vert \lambda_i \vert}) x_{0,i} {\;}_{(1 \leq i \leq n-2)},\\
&\frac{1}{2b_0} \Big(2 b_0 b' - \sum_{i=1}^{n-2} b_i^2 + b_0^2 \sum_{i=1}^{n-2} \lambda_i x_{0,i}^2 \Big) + \sum_{i=1}^{n-2} \frac{b_i^2 - b_0^2 \lambda_i x_{0,i}^2}{4 b_0^2 \sqrt{\vert \lambda_i \vert}} \sh (2b_0 \sqrt{\vert \lambda_i \vert})\\
&+ \sum_{i=1}^{n-2} \frac{b_i x_{0,i}}{2 b_0} \ch (2b_0 \sqrt{\vert \lambda_i \vert}) + \Big( v_0 - \frac{1}{2b_0} \sum_{i=1}^{n-2} b_i x_{0,i} \Big) \Big)
\end{align*}
for any $\bar{b} = (b_0, b_1, ..., b_{n-2}, b') \in T_y M \simeq \R^n$ such that $b_0 \neq 0$ and
$$\mathrm{Exp}_y (0, b_1, ..., b_{n-2},b') = (t_0, b_1 + x_{0,1}, ..., b_{n-2} + x_{0,n-2}, b' + v_0).$$
\end{proposition}

\begin{proof}
The geodesic equations corresponding to the Levi-Civita connection on $M$ are given by
$$\begin{cases}
\frac{d^2 t}{ds^2} (s) = 0,\\
\frac{d^2 v}{ds^2} (s) +2 \sum_{i=1}^{n-2} \lambda_i x_i (s) \frac{dt}{ds} (s) \frac{dx_i}{ds} (s) = 0,\\
\frac{d^2 x_i}{ds^2} (s) + \lambda_i x_i (s) \big( \frac{dt}{ds} (s) \big)^2 = 0
\end{cases}.$$
The solutions $\gamma (s) := (t(s), x_1 (s), ..., x_{n-2} (s), v(s))$ of these equations with initial data
\begin{align*}
&\gamma (0) = (t(0), x_1 (0), ..., x_{n-2} (0), v (0)) = (t_0, x_{0,1}, ..., x_{0,n-2}, v_0),\\
&\dot{\gamma} (0) = \Big( \frac{dt}{ds} (0), \frac{dx_1}{ds} (0), ..., \frac{dx_{n-2}}{ds} (0), \frac{dv}{ds} (0) \Big) = (b_0, b_1, ..., b_{n-2}, b')
\end{align*}
are given by $\mathrm{Exp}_{\gamma (0)} (s \dot{\gamma} (0))$ as above.
\end{proof}

\begin{corollary}
Let $(M,s,\hat{g})$ be the Cahen-Wallach space corresponding to parameters $\lambda_1, ..., \lambda_{n-2} \in \R_0$.
For any point $y = (t_0, x_0, v_0) \in M$, the exponential mapping at $y$ is a diffeomorphism from the open subset
$$\lbrace (b_0, b_1, ..., b_{n-2}, b') \in T_y M \mid \forall i \in \lbrace 1, ..., n-2 \rbrace, \; \lambda_i > 0 \Rightarrow b_0 \sqrt{\lambda_i} \notin \pi \mathbb{Z}_0 \rbrace$$
of $T_y M$ to its image given by
$$\lbrace (t, x, v) \in M \mid \forall i \in \lbrace 1, ..., n-2 \rbrace, \; \lambda_i > 0 \Rightarrow (t-t_0) \sqrt{\lambda_i} \notin \pi \mathbb{Z}_0 \rbrace.$$
The determinant of the differential of $\mathrm{Exp}_y$ at any vector $(b_0, b_1, ..., b_{n-2}, b') \in T_y M$ is equal to
$$\prod_{i=1}^{n-2} \Big\vert \frac{\sh (b_0 \sqrt{\vert \lambda_i \vert})}{b_0 \sqrt{\vert \lambda_i \vert}}\Big\vert \; \text{if } b_0 \neq 0, \text{ and } 1 \; \text{otherwise.}$$
\end{corollary}


\subsection{The orbits of the isotropy group and the orbital integrals}

Let us fix $(0, ..., 0)$ as base-point of $M$. Then we get the following identification $T_{0} M \simeq \mathfrak{p}$
$$\begin{cases}
T_{0} M \simeq \R^n &\rightarrow \mathfrak{p}\\
(b_0, b_1, ..., b_{n-2}, b') &\mapsto b_0 U + \sum_{i=1}^{n-2} b_i W_i + b' Z
\end{cases}.$$

\begin{proposition}\label{proporbsolvable}
The orbits under the adjoint action of the isotropy group $K$ in $\mathfrak{p} \simeq T_{0} M$ are
\begin{align*}
&\lbrace (b_0, b_1, ..., b_{n-2}, b') \in T_{0} M \mid 2 b_0 b' - \sum_{i=1}^{n-2} b_i^2 = \alpha, \; b_0 = \beta \rbrace =: \Sigma_{\alpha, \beta},\\
&\lbrace (0, \gamma_1, ..., \gamma_{n-2}, b') \mid b' \in \R \rbrace,\\
&\lbrace (0, 0, ..., 0, \beta') \rbrace
\end{align*}
for $\beta \in \R_0, \alpha \in \R, (\gamma_1, ..., \gamma_{n-2}) \in \R^{n-2} \setminus \lbrace 0 \rbrace$ and $\beta' \in \R$.
\end{proposition}

\begin{proof}
The adjoint action of $K$ on $\mathfrak{p} \simeq \R^n$ is given by
\begin{align*}
&\mathrm{Ad} (\exp (p_1 K_1 + ... + p_{n-2} K_{n-2})) (b_0, b_1, ..., b_{n-2}, b')\\
&= \Big( b_0, b_1 - \lambda_1 p_1 b_0, ..., b_{n-2} - \lambda_{n-2} p_{n-2} b_0, b' + \frac{b_0}{2} \sum_{i=1}^{n-2} \lambda_i^2 p_i^2 - \sum_{i=1}^{n-2} \lambda_i p_i b_i \Big)
\end{align*}
for $(p_1, ..., p_{n-2}) \in \R^{n-2}$ and $(b_0, b_1, ..., b_{n-2}, b') \in T_{0} M \simeq \mathfrak{p}$. Once $(b_0, b_1, ..., b_{n-2}, b')$ is fixed, we must distinguish the cases when $b_0 \neq 0$ and $b_0 = 0$ to identify the different isotropy orbits.
\end{proof}

Note that the stabilizer of any point $X = (b_0, b_1, ..., b_{n-2}, b') \in T_0 M \simeq \R^n$ such that $b_0 \neq 0$ under the action of the isotropy group $K$ is just the neutral element. Therefore, the $K$-orbit of $X$ is diffeomorphic to $K$.

\begin{definition}
For any $f \in \mathcal{C}_c (M)$ and any point $y = (t_0, x_0, v_0) \in M$, the \emph{orbital integrals} of $f$ at $y$ are given by
$$(M^X f) (y) := \int_K f(gk.\mathrm{Exp}_0 (X))\; dk$$
where $g \in G$ such that $y = g.0$, $X = (b_0, b_1, ..., b_{n-2},b') \in T_0 M \simeq \R^n$ such that $b_0 \neq 0$ and if $\lambda_i > 0$, $b_0 \sqrt{\lambda_i} \notin \pi \mathbb{Z}$, and $dk$ is the invariant measure $(\prod_{i=1}^{n-2} \vert \lambda_i \vert)\, dp_1 ... dp_{n-2}$ on $K \simeq \R^{n-2}$.
\end{definition}

The orbital integrals of a function $f \in \mathcal{C}_c (M)$ at $y \in M$ are integrals of $f$ over the orbits $\mathrm{Exp}_y (\Sigma_{\pm r^2, b_0} (y))$ where
\begin{align*}
\Sigma_{\pm r^2, b_0} (y) := \Big\lbrace &(\tilde{b}_0, b_1, ..., b_{n-2}, b') \in T_y M \mid \tilde{b}_0 = b_0,\\
&\sum_{i=1}^{n-2} \lambda_i x_{0,i}^2 b_0^2 + 2 b_0 b' - \sum_{i=1}^{n-2} b_i^2 = \pm r^2 \Big\rbrace
\end{align*}
for $r > 0$ and $b_0 \in \R_0$ such that if $\lambda_i > 0$, $b_0 \sqrt{\lambda_i} \notin \pi \mathbb{Z}_0$.
We use the following notation for the orbital integrals
\begin{align*}
&(M^{r,b_0}_+ f) (y) := (M^X f) (y), \quad \text{where } X = (b_0, 0, ..., 0, r^2/2b_0) \in T_0 M \simeq \R^n,\\
&(M^{r,b_0}_- f) (y) := (M^X f) (y), \quad \text{where } X = (b_0, 0, ..., 0, -r^2/2b_0) \in T_0 M \simeq \R^n,
\end{align*}
for $r > 0$ and $b_0 \in \R_0$ such that if $\lambda_i > 0$, $b_0 \sqrt{\lambda_i} \notin \pi \mathbb{Z}_0$.

\begin{lemma}\label{lemmaorbsolvable}
For any $f \in \mathcal{C}_c (M)$, $y = g.0 \in M$, $r > 0$ and $b_0 \in \R_0$ such that if $\lambda_i > 0$, $b_0 \sqrt{\lambda_i} \notin \pi \mathbb{Z}_0$,
\begin{align*}
(M^{r, b_0}_\pm f) (y) = \frac{1}{\vert b_0 \vert^{n-2}} \int_{\R^{n-2}} &f\Big( g.\mathrm{Exp}_{0} \Big( b_0, b_1, ..., b_{n-2}, \frac{1}{2b_0} \big( \pm r^2 + \sum_{i=1}^{n-2} b_i^2 \big) \Big) \Big)\\
&db_1 ... db_{n-2}.
\end{align*}
\end{lemma}

\begin{proof}
Using the expression of the action of $K$ on $T_0 M$ given in the proof of proposition \ref{proporbsolvable}, we get
\begin{align*}
(M^{r, b_0}_\pm f) (y) = \Big(\prod_{i=1}^{n-2} \vert \lambda_i \vert \Big) \int_{\R^{n-2}} &f\Big( g.\mathrm{Exp}_{0} \Big( b_0, - \lambda_1 p_1 b_0, ..., - \lambda_{n-2} p_{n-2} b_0,\\
&\frac{1}{2b_0} \big( \pm r^2 + \sum_{i=1}^{n-2} \lambda_i^2 p_i^2 b_0^2 \big) \Big) \Big)\; dp_1 ... dp_{n-2}.
\end{align*}
Then a simple change of variables in the integral yields the result.
\end{proof}


\section{Limit formulas on solvable Lorentzian symmetric spaces}\label{sectionsolvableorbital}

Let $(M,s,\hat{g})$ be the Cahen-Wallach space corresponding to parameters $\lambda_1, ..., \lambda_{n-2} \in \R_0$. In order to determine a function in terms of its orbital integrals on $M$, we consider invariant differential operators.

\subsection{Invariant differential operators}

Let us recall that $M$ is $G$-homogeneous space for the connected, simply connected Lie group $G$ presented in section \ref{sectionsolvable}.

\begin{proposition}\label{propinvarop}
The Laplace-Beltrami operator on $M$ associated to the Lorentzian $\hat{g} = \sum_{i=1}^{n-2} \lambda_i x_i^2 dt \otimes dt + dt \otimes dv + dv \otimes dt - \sum_{i=1}^{n-2} dx_i \otimes dx_i$
$$L := \frac{\partial^2}{\partial t \partial v} + \frac{\partial^2}{\partial v \partial t} - \sum_{i=1}^{n-2} \lambda_i x_i^2 \frac{\partial^2}{\partial v^2} - \sum_{i=1}^{n-2} \frac{\partial^2}{\partial x_i^2}$$
as well as the order-one differential operator $$\frac{\partial}{\partial v}$$ are $G$-invariant.
\end{proposition}

At any point $y = (t_0, x_0, v_0) = g.0 \in M$, we consider coordinates $(r, b_0, b_1, ..., b_{n-2})$ on either $\lbrace \mathrm{Exp}_y (\bar{b}) \mid \bar{b} \in T_y M, \; \hat{g}_y (\bar{b}, \bar{b}) > 0 \rbrace$ or $\lbrace \mathrm{Exp}_y (\bar{b}) \mid \bar{b} \in T_y M, \; \hat{g}_y (\bar{b}, \bar{b}) < 0 \rbrace$ which are defined by
$$(r, b_0, b_1, ..., b_{n-2}) \mapsto g.\mathrm{Exp}_{0} \Big( b_0, b_1, ..., b_{n-2}, \frac{1}{2b_0} \big( \varepsilon \, r^2 + \sum_{i=1}^{n-2} b_i^2 \big) \Big),$$
where $\epsilon = 1$ or $-1$ respectively. In this coordinates, the expression of the two invariant differential operator in proposition \ref{propinvarop} is
\begin{align*}
&L = \varepsilon \, \Big\lbrace \frac{1}{r} \frac{\partial}{\partial r} \Big( r \frac{\partial}{\partial r} \Big) + \frac{1}{r} \frac{\partial}{\partial r} \Big( b_0 \frac{\partial}{\partial b_0} \Big) + \frac{1}{\frac{r}{\vert b_0 \vert} \prod_{i=1}^{n-2} \frac{\sh (b_0 \sqrt{\vert \lambda_i \vert})}{\sqrt{\vert \lambda_i \vert}}}\\
&\quad \quad \quad \frac{\partial}{\partial b_0} \Big( \frac{b_0}{\vert b_0 \vert} \prod_{i=1}^{n-2} \frac{\sh (b_0 \sqrt{\vert \lambda_i \vert})}{\sqrt{\vert \lambda_i \vert}} \frac{\partial}{\partial r}\Big) \Big\rbrace + L_{\mathrm{Exp}_y  (\Sigma_{\varepsilon \, r^2, b_0} (y))},\\
&\frac{\partial}{\partial v} = \varepsilon \, \frac{b_0}{r} \frac{\partial}{\partial r},
\end{align*}
where $L_{\mathrm{Exp}_y  (\Sigma_{\varepsilon \, r^2, b_0} (y))}$ is the Laplace-Beltrami operator on the isotropy orbit associated to the induced metric and $\epsilon$ is either $1$ or $-1$ depending on whether the domain is $\lbrace \mathrm{Exp}_y (\bar{b}) \mid \bar{b} \in T_y M, \; \hat{g}_y (\bar{b}, \bar{b}) > 0 \rbrace$ or $\lbrace \mathrm{Exp}_y (\bar{b}) \mid \bar{b} \in T_y M, \; \hat{g}_y (\bar{b}, \bar{b}) < 0 \rbrace$.

\begin{proposition}\label{propintegrals}
For any $f \in \mathcal{C}_c^\infty (M)$, $y \in M$, $r > 0$ and $b_0 \in \R_0$ such that if $\lambda_i > 0$, $b_0 \sqrt{\lambda_i} \notin \pi \mathbb{Z}_0$,
\begin{enumerate}
\item $L (M^{r, b_0}_\pm f) (y) = M^{r,b_0}_\pm (Lf) (y)$,
\item $\frac{\partial}{\partial v} (M^{r,b_0}_\pm f) (y) = M^{r,b_0}_\pm \Big( \frac{\partial f}{\partial v} \Big) (y)$,
\item $M^{r,b_0}_\pm \Big( \frac{\partial f}{\partial v} \Big) (y) = \pm \frac{b_0}{r} \frac{\partial}{\partial r} (M^{r,b_0}_\pm f) (y)$.
\end{enumerate}
\end{proposition}

\begin{proof}
By a theorem stated in S. Helgason's book \cite{Hel59}, for every $G$-invariant differential operator $D$ on $M$, there exists a bi-invariant differential operator $\tilde{D}$ on $G$ such that
$$\forall f \in \mathcal{C}^\infty (M), \; \tilde{D} (f \circ \pi) = (Df) \circ \pi,$$
where $\pi : G \rightarrow M$ is the canonical projection. Therefore, when we apply $D$ to the orbital integrals of a function $f \in \mathcal{C}^\infty_c (M)$, we get
\begin{align*}
D &(M^{r,b_0}_\pm f) (g.0) = \tilde{D} \big( (M^{r,b_0}_\pm f) \circ \pi \big) (g)\\
&= \tilde{D} \Big\lbrace \frac{1}{\vert b_0 \vert^{n-2}} \int_{\R^{n-2}} f\Big(g.\mathrm{Exp}_{0} \Big(b_0, b_1, ..., b_{n-2}, \frac{1}{2b_0} \big( \pm r^2 + \sum_{i=1}^{n-2} b_i^2 \big) \Big)\Big)\\
& \qquad db_1... db_{n-2} \Big\rbrace\\
&= \frac{1}{\vert b_0 \vert^{n-2}} \int_{\R^{n-2}} (D f) \Big(g.\mathrm{Exp}_{0} \Big(b_0, b_1, ..., b_{n-2}, \frac{1}{2b_0} \big( \pm r^2 + \sum_{i=1}^{n-2} b_i^2 \big) \Big)\Big)\\
& \qquad db_1... db_{n-2}
\end{align*}
because $\tilde{D}$ is bi-invariant. This implies the points $1$ and $2$. For point $3$, we use the local expression of $\frac{\partial}{\partial v}$.
\end{proof}

\subsection{One-parameter generalized Riesz potentials}

Let us focus on the first set of orbital integrals, namely $(M^{r,b_0}_+ f)$. The corresponding one-parameter generalized Riesz potentials are then given below.

\begin{definition}
For any $f \in \mathcal{C}_c (M)$, $y \in M$, $b_0 \in \R_0$ such that if $\lambda_i > 0$, $b_0 \sqrt{\lambda_i} \notin \pi \mathbb{Z}_0$ and $\mu \in \C$ such that $\mathrm{Re} (\mu) > n-1$,
$$(J^{\mu, b_0}_+ f) (y) := \frac{1}{H_{n-1} (\mu)} \int_0^\infty (M^{r,b_0}_+ f) (y) r^{\mu - n +2} dr$$
where $H_{n-1} (\mu) := \pi^{(n-3)/2} 2^{\mu -1} \Gamma (\mu/2) \Gamma ((\mu + 3 - n) / 2)$.
\end{definition}

\begin{remark}\label{generalRieszpot}
For any $y = g.0 \in M$ and fixed $b_0 \in \R_0$ such that if $\lambda_i > 0$, $b_0 \sqrt{\lambda_i} \notin \pi \mathbb{Z}_0$,
\begin{align*}
(J^{\mu, b_0}_+ f) (y) = &\frac{1}{\vert b_0 \vert^{n-2}} \frac{1}{H_{n-1} (\mu)} \int_{(D_0^+)_{cc}} F_{g, b_0} (u_1, ..., u_{n-1})\\
&u_{n-1} \Big( \sqrt{u_{n-1}^2 - u_{n-2}^2 - ... - u_1^2} \Big)^{\mu - n + 1} du_1 ... du_{n-1}
\end{align*}
where $F_{g,b_0} (u_1, ..., u_{n-1}) := f \big(g.\mathrm{Exp}_0 \big(b_0, u_1, ..., u_{n-2}, \frac{u_{n-1}^2}{2b_0}\big) \big)$ and
$$(D_0^+)_{cc} := \lbrace (u_1, ..., u_{n-1}) \in \R^{n-1} \mid u_{n-1}^2 - u_{n-2}^2 - ... - u_1^2 > 0, \; u_{n-1} > 0 \rbrace.$$
This means that the one-parameter generalized Riesz potentials on $M$ turn out to be some flat Lorentzian Riesz potentials on $(\R^{n-1}, \mathbb{I}_{1,n-2})$.
\end{remark}

\begin{proposition}\label{propgeneralRiesz}
For any $f \in \mathcal{C}_c^\infty (M)$, $y = g.0 \in M$ and fixed $b_0 \in \R_0$ such that if $\lambda_i > 0$, $b_0 \sqrt{\lambda_i} \notin \pi \mathbb{Z}_0$,
\begin{enumerate}
\item $(J^{\mu,b_0}_+ f) (y)$ holomorphically extends to the whole complex plane $\C$ with respect to $\mu$,
\item $\underset{\mu \rightarrow 0}{\lim}\; (J^{\mu, b_0}_+ f) (y) = \frac{1}{\vert b_0 \vert^{n-2}} F_{g,b_0} (0, ..., 0) .0 = 0$,
\item $\frac{\partial}{\partial v} (J^{\mu, b_0}_+ f) (y) = J^{\mu, b_0}_+ \Big( \frac{\partial f}{\partial v} \Big) (y) = - \frac{b_0}{\mu - 2} (J^{\mu - 2, b_0}_+ f) (y)$.
\end{enumerate}
\end{proposition}

\begin{proof}
Property $1$ is true thanks to remark \ref{generalRieszpot} and theorem \ref{flatRieszpot} which states that flat Lorentzian Riesz potentials can be holomorphically extended on the whole complex plane with respect to the parameter.
Property $2$ also follows from theorem \ref{flatRieszpot} thanks to remark \ref{generalRieszpot}.
Finally, we use the local expression of the invariant differential operator 
$$\frac{\partial}{\partial v} = \frac{b_0}{r} \frac{\partial}{\partial r}$$
on a dense open subset of $\lbrace \mathrm{Exp}_y (\bar{b}) \mid \bar{b} \in T_y M, \, \hat{g}_y (\bar{b}, \bar{b}) > 0 \rbrace$ to get property $3$ whenever $\mathrm{Re} (\mu) > n-1$. Furthermore, it remains true for every $\mu \in \C$ thanks to the uniqueness of the holomorphic extensions.
\end{proof}

\begin{corollary}
For any $f \in \mathcal{C}^\infty_c (M)$, $y \in M$, $b_0 \in \R_0$ such that if $\lambda_i > 0$, $b_0 \sqrt{\lambda_i} \notin \pi \mathbb{Z}_0$ and any $k \in \mathbb{N}_0$,
$$\frac{\partial^k}{\partial v^k} (J^{\mu, b_0}_+ f) (y) = J^{\mu, b_0}_+ \Big( \frac{\partial^k f}{\partial v^k} \Big) (y) = \frac{(-b_0)^k}{(\mu - 2) (\mu-4) ... (\mu - 2k)} (J^{\mu-2k, b_0}_+ f) (y).$$
\end{corollary}

\begin{proposition}
For any $f \in \mathcal{C}^\infty_c (M)$, $y \in M$ and $b_0 \in \R_0$ such that if $\lambda_i > 0$, $b_0 \sqrt{\lambda_i} \notin \pi \mathbb{Z}_0$,
$$J^{1, b_0}_+ \Big( \frac{\partial f}{\partial v} \Big) (y) = \underset{\mu \rightarrow -1}{\lim}\; \frac{-b_0}{\mu} (J^{\mu, b_0}_+ f) (y) = \frac{-b_0 \Gamma \big(\frac{n-1}{2}\big)}{\vert b_0 \vert^{n-2} \sqrt{\pi} \Gamma \big( \frac{n}{2} \big)} f(g.\mathrm{Exp}_0 (b_0, 0, ..., 0)).$$
\end{proposition}

\begin{proof}
We use property $3$ in proposition \ref{propgeneralRiesz} and the following change of variables in the expression in remark \ref{generalRieszpot}
$$(u_1, ..., u_{n-1}) = (r \sinh (\zeta) \omega (\theta_1, ..., \theta_{n-3}), r \cosh (\zeta))$$
where $\zeta \in ]0, +\infty[$ and $\omega (\bar{\theta}) \in \mathbb{S}^{n-3}$, to get
\begin{align*}
J^{1,b_0}_+ \Big(\frac{\partial f}{\partial v} \Big) (y) = \underset{\mu \rightarrow -1}{\lim}\; &\frac{-b_0}{\vert b_0 \vert^{n-2} \mu H_{n-1} (\mu)} \int_0^\infty \int_0^\infty \int_{\mathbb{S}^{n-3}} F_{g,b_0} (r\sinh (\zeta) \omega (\bar{\theta}),\\ &r \cosh (\zeta))\; r^\mu \cosh (\zeta) \sinh^{n-3} (\zeta) d\omega (\bar{\theta}) d\zeta dr.
\end{align*}
Then the consecutive changes $T = e^{-2\zeta}$ and $r = \sigma \sqrt{T}$ lead to
\begin{align*}
\underset{\mu \rightarrow -1}{\lim}\; &\frac{-b_0 \Gamma \big(\frac{\mu + 1}{2} \big)}{\vert b_0 \vert^{n-2} \pi^{(n-2)/2} \Gamma (\mu +1) \Gamma \big( \frac{\mu + 3 - n}{2}\big)} \int_0^1 \int_0^\infty  \int_{\mathbb{S}^{n-3}} F_{g,b_0} \Big( \frac{\sigma}{2} (1-T) \omega (\bar{\theta}),\\
&\frac{\sigma}{2} (1 + T) \Big)\; \sigma^{(\mu + 1)-1} 2^{1-n} T^{(\mu + 1 - n)/2} (1+T) (1-T)^{n-3} d\omega (\bar{\theta}) d\sigma dT
\end{align*}
which can be seen as a double Riemann-Liouville integral. Thanks to arguments found in Riesz's paper \cite{Rie49}, the limit is equal to
\begin{align*}
\underset{\nu \rightarrow \frac{2-n}{2}}{\lim}\; \frac{-b_0 2^{1-n} \Omega_{n-2} \Gamma \big( \nu + \frac{n-2}{2}\big)}{\vert b_0 \vert^{n-2} \pi^{(n-2)/2} \Gamma (\nu)} F_{g,b_0} (0, ..., 0) \int_0^1 T^{\nu -1} (1+T) (1-T)^{n-3} dT
\end{align*}
where $\Omega_{n-2}$ is the area of the $(n-3)$-dimensional standard sphere.
Using Euler's beta functions to rewrite the remaining integral and since $F_{g,b_0} (0, ..., 0) = f(g.\mathrm{Exp}_0 (b_0, 0, ..., 0))$, we get the expected equality.
\end{proof}

\begin{corollary}\label{corollaryrecup}
For any $f \in \mathcal{C}^\infty_c (M)$, $y = g.0 \in M$, $b_0 \in \R_0$ such that if $\lambda_i > 0$, $b_0 \sqrt{\lambda_i} \notin \pi \mathbb{Z}_0$ and any $k \in \mathbb{N}_0$,
$$J^{2k-1, b_0}_+ \Big( \frac{\partial^k f}{\partial v^k} \Big) (y) = \frac{\Gamma \big( \frac{n-1}{2} \big)}{2^{k-1} \Gamma \big( \frac{n}{2} \big)} \frac{(-b_0)^k}{\Gamma \big( \frac{2k-1}{2}\big) \vert b_0 \vert^{n-2}} f(g.\mathrm{Exp}_0 (b_0, 0, ..., 0)).$$
\end{corollary}

\subsection{Limit formulas for the orbital integrals}

\begin{lemma}\label{lemmalimitorbit}
For any function $f \in \mathcal{C}^\infty_c (M)$, any $b_0 \in \R_0$ such that if $\lambda_i > 0$, $b_0 \sqrt{\lambda_i} \notin \pi \mathbb{Z}_0$ and any $y \in M$, the limit
$$\underset{r \overset{>}{\rightarrow} 0}{\lim}\; (M^{r,b_0}_\pm f) (y)$$ exists whenever $n = \dim (M) > 2$. 
\end{lemma}

\begin{proof}
By lemma \ref{lemmaorbsolvable}, if $y = g.0$, the orbital integrals $(M^{r,b_0}_\pm f) (y)$ writes
$$\frac{1}{\vert b_0 \vert^{n-2}} \int_{\R^{n-2}} f\Big(g.\mathrm{Exp}_{0} \Big(b_0, b_1, ..., b_{n-2}, \frac{1}{2b_0} \big( \pm r^2 + \sum_{i=1}^{n-2} b_i^2 \big) \Big)\Big) db_1... db_{n-2}.$$
In the first series of integrals, namely $(M^{r,b_0}_+ f)$, we apply the change of coordinates
$$(b_1, ..., b_{n-2}) = r \sinh (\zeta) \omega (\theta_1, ..., \theta_{n-3})$$
where $\zeta \in ]0, +\infty[$ and $\omega (\bar{\theta}) \in \mathbb{S}^{n-3}$, to get
\begin{align*}
\frac{1}{\vert b_0 \vert^{n-2}} \int_0^\infty \int_{\mathbb{S}^{n-3}} &f\Big(g.\mathrm{Exp}_{0} \Big(b_0, r \sinh (\zeta) \omega (\bar{\theta}), \frac{r^2}{2b_0} \cosh^2 (\zeta)\Big)\Big)\\
& r^{n-2} \sinh^{n-3} (\zeta) \cosh (\zeta) d\omega (\bar{\theta}) d\zeta.
\end{align*}
Next, we define
$F_{g,b_0} (u_1, ..., u_{n-1}) := f\big( g.\mathrm{Exp}_{0} \big( b_0, u_1, ..., u_{n-2}, \frac{u_{n-1}^2}{2b_0} \big) \big)$ and set $t = r \sinh (\zeta)$ in the integral.
Thus the orbital integrals write
$$(M^{r,b_0}_+ f) (g.0) = \frac{1}{\vert b_0 \vert^{n-2}} \int_0^\infty \int_{\mathbb{S}^{n-3}} F_{g,b_0} \big(t \omega (\bar{\theta}), \sqrt{r^2 + t^2}\big) t^{n-3} d\omega (\bar{\theta}) dt$$
and it is clear that the limit when $r$ goes to zero exists.
This remains true with the other series of orbital integrals, namely $(M^{r,b_0}_- f)$.
\end{proof}

\begin{corollary}\label{corollaryforlimorbit}
Assume $n = \dim (M) > 3$. Then, for any function $f \in \mathcal{C}_c (M)$, any $y \in M$ and fixed $b_0 \in \R_0$ such that if $\lambda_i > 0$, $b_0 \sqrt{\lambda_i} \notin \pi \mathbb{Z}_0$,
$$(J^{n-3, b_0}_+ f) (y) = \frac{1}{(4\pi)^{(n-3)/2} \Gamma \big( \frac{n-3}{2}\big)}\; \underset{r \overset{>}{\rightarrow} 0}{\lim}\; (M^{r,b_0}_+ f) (y).$$
\end{corollary}

\begin{proof}
The definition of one-parameter generalized Riesz potentials gives us
$$(J^{n-3, b_0}_+ f) (y) = \underset{\mu \rightarrow n-3}{\lim}\; \frac{\Gamma \big( \frac{\mu + 4 - n}{2}\big)}{2^{n-3} \pi^{(n-2)/2} \Gamma \big( \frac{\mu}{2}\big)} \frac{1}{\Gamma (\mu + 3 - n)} \int_0^\infty \tilde{F}_{b_0} (r) r^{(\mu + 3 - n) -1} dr$$
where $\tilde{F}_{b_0} (r) := (M^{r,b_0}_+ f) (y)$.
Since $\tilde{F}_{b_0}$ is a compactly supported, continuous function whose limit when $r$ goes to zero exists thanks to lemma \ref{lemmalimitorbit}, the limit of the Riemann-Liouville integral above is equal to
$$\frac{\sqrt{\pi}}{2^{n-3} \pi^{(n-2)/2} \Gamma \big( \frac{n-3}{2} \big)}\; \underset{r \overset{>}{\rightarrow} 0}{\lim}\; \tilde{F}_{b_0} (r).$$
\end{proof}

\begin{theorem}
If $n = \dim (M) > 2$ is even, for any $f \in \mathcal{C}_c^\infty (M)$ and $y \in M$,
$$f(y) = \frac{\Gamma \big( \frac{n}{2} \big)}{\pi^{(n-3)/2} \Gamma \big(\frac{n-1}{2}\big)}\; \underset{b_0 \overset{\neq}{\rightarrow} 0}{\lim}\; \underset{r \overset{>}{\rightarrow} 0}{\lim}\; \Big\lbrace \frac{- b_0^2}{2r} \frac{\partial}{\partial r} \Big\rbrace^{(n-2)/2} (M^{r, b_0}_+ f) (y).$$
\end{theorem}

\begin{proof}
Since $n = \dim (M) > 2$ is even, there exists $k \in \mathbb{N}_0$ such that $n = 2k + 2$. Let us first fix $b_0 \in \R_0$ such that if $\lambda_i > 0$, $-\pi / \sqrt{\lambda_i} < b_0 < \pi / \sqrt{\lambda_i}$. By corollary \ref{corollaryforlimorbit},
$$J^{n-3, b_0}_+ \Big( \frac{\partial^k f}{\partial v^k} \Big) (y) = \frac{1}{(4\pi)^{(n-3)/2} \Gamma \big( \frac{n-3}{2} \big)} \; \underset{r \overset{>}{\rightarrow} 0}{\lim} \; M^{r,b_0}_+ \Big( \frac{\partial^k f}{\partial v^k} \Big) (y).$$
Furthermore, thanks to property $3$ in proposition \ref{propintegrals}, we get
$$J^{n-3, b_0}_+ \Big( \frac{\partial^k f}{\partial v^k} \Big) (y) = \frac{1}{(4\pi)^{(n-3)/2} \Gamma \big( \frac{n-3}{2} \big)} \; \underset{r \overset{>}{\rightarrow} 0}{\lim} \; \Big\lbrace \frac{b_0}{r} \frac{\partial}{\partial r}\Big\rbrace^k (M^{r,b_0}_+ f) (y).$$
On the other hand, by corollary \ref{corollaryrecup}, since $n-3 = 2k-1$,
$$J^{n-3, b_0}_+ \Big( \frac{\partial^k f}{\partial v^k} \Big) (y) = \frac{\Gamma \big( \frac{n-1}{2} \big)}{2^{(n-4)/2} \Gamma \big( \frac{n}{2} \big)} \frac{(-b_0)^{(n-2)/2}}{\Gamma \big( \frac{n-3}{2}\big) \vert b_0 \vert^{n-2}}\; f(g.\mathrm{Exp}_0 (b_0, 0, ..., 0)),$$
hence the result.
\end{proof}

The same kind of limit formula holds for the other series of orbital integrals, namely $(M^{r,b_0}_- f)$. In other words, the function $f$ is also determined in terms of $(M^{r,b_0}_- f)$. 

\begin{theorem}
If $n = \dim (M) > 2$ is even, for any $f \in \mathcal{C}_c^\infty (M)$ and $y \in M$,
$$f(y) = \frac{\Gamma \big( \frac{n}{2} \big)}{\pi^{(n-3)/2} \Gamma \big(\frac{n-1}{2}\big)}\; \underset{b_0 \overset{\neq}{\rightarrow} 0}{\lim}\; \underset{r \overset{>}{\rightarrow} 0}{\lim}\; \Big\lbrace \frac{- b_0^2}{2r} \frac{\partial}{\partial r} \Big\rbrace^{(n-2)/2} (M^{r, b_0}_- f) (y).$$
\end{theorem}

This limit formula is obtained in the same way by considering the corresponding one-parameter generalized Riesz potentials $$(J^{\mu, b_0}_- f) (y) := \frac{1}{H_{n-1} (\mu)} \int_0^\infty (M^{r,b_0}_- f) (y) r^{\mu - n + 2} dr.$$

Finally, we deal with the odd-dimensional case.

\begin{theorem}
If $n = \dim (M) \geq 3$ is odd, for any $f \in \mathcal{C}^\infty_c (M)$ and $y \in M$,
$$f(y) = \frac{\Gamma \big(\frac{n}{2}\big)}{\pi^{(n-2)/2} \Gamma \big( \frac{n-1}{2}\big)}\; \underset{b_0 \overset{\neq}{\rightarrow} 0}{\lim}\; \int_0^\infty \Big\lbrace \frac{-b_0}{r} \frac{\partial}{\partial r} \Big\rbrace \Big\lbrace \frac{-b_0^2}{2r} \frac{\partial}{\partial r} \Big\rbrace^{(n-3)/2} (M^{r,b_0}_+ f) (y) dr.$$
In the same way,
$$f(y) = \frac{\Gamma \big(\frac{n}{2}\big)}{\pi^{(n-2)/2} \Gamma \big( \frac{n-1}{2}\big)}\; \underset{b_0 \overset{\neq}{\rightarrow} 0}{\lim}\; \int_0^\infty \Big\lbrace \frac{-b_0}{r} \frac{\partial}{\partial r} \Big\rbrace \Big\lbrace \frac{-b_0^2}{2r} \frac{\partial}{\partial r} \Big\rbrace^{(n-3)/2} (M^{r,b_0}_- f) (y) dr.$$
\end{theorem}

\begin{proof}
Since $n = \dim (M) \geq 3$ is odd, there exists $k \in \mathbb{N}_0$ such that $n = 2k+1$. Let us first fix $b_0 \in \R_0$ such that if $\lambda_i > 0$, $-\pi / \sqrt{\lambda_i} < b_0 < \pi / \sqrt{\lambda_i}$. By corollary \ref{corollaryrecup}, since $n-2 = 2k-1$,
$$J^{n-2, b_0}_+ \Big( \frac{\partial^k f}{\partial v^k} \Big) (y) = \frac{\Gamma \big( \frac{n-1}{2} \big)}{2^{(n-3)/2} \Gamma \big( \frac{n}{2} \big)} \frac{(-b_0)^{(n-1)/2}}{\Gamma \big( \frac{n-2}{2}\big) \vert b_0 \vert^{n-2}}\; f(g.\mathrm{Exp}_0 (b_0, 0, ..., 0)).$$
Furthermore, by definition of the one-parameter generalized Riesz potentials,
$$J^{n-2, b_0}_+ \Big( \frac{\partial^k f}{\partial v^k}\Big) (y) = \frac{1}{H_{n-1} (n-2)} \int_0^\infty M^{r,b_0}_+ \Big( \frac{\partial^k f}{\partial v^k} \Big) (y) dr$$
where the integral converges due to lemma \ref{lemmalimitorbit} and because $f$ is compactly supported.
Thanks to property $3$ in proposition \ref{propintegrals},
$$J^{n-2, b_0}_+ \Big( \frac{\partial^k f}{\partial v^k}\Big) (y) = \frac{1}{2^{n-3} \pi^{(n-2)/2} \Gamma \big( \frac{n-2}{2} \big)} \int_0^\infty \Big\lbrace \frac{b_0}{r} \frac{\partial}{\partial r} \Big\rbrace^k (M^{r,b_0}_+ f) (y) dr,$$
hence the result.
It works the same way for the other series of orbital integrals, namely $(M^{r,b_0}_- f)$.
\end{proof}

\section*{Acknowledgements}

I thank my PhD advisor, Simone Gutt, for her support and mathematical help and Michel Cahen for many fruitful discussions. I also thank the "Fonds de la Recherche Scientifique" for funding my research through a "mandat d'Aspirant".

\end{document}